\documentclass[11pt, reqno]{amsart}
\usepackage{changes}
\usepackage{amssymb, amsmath, amsthm, mathrsfs, MnSymbol}
\usepackage{hyperref}
\usepackage{enumitem}
\usepackage[most]{tcolorbox}
\usepackage{xcolor}
\usepackage{float}
\usepackage{tikz}

\newtheorem*{theoA}{Theorem A}
\newtheorem*{theoB}{Theorem B}
\newtheorem*{theoC}{Theorem C}
\newtheorem*{theoD}{Theorem D}
\newtheorem*{theoE}{Theorem E}

\newtheorem{theo}{Theorem}[section]
\newtheorem{lem}{Lemma}[section]

\newtheorem{exm}{Example}[section]

\newtheorem{rem}{Remark}[section]
\newcommand{\ol}{\overline}

\newcommand{\be}{\begin{equation}}
\newcommand{\ee}{\end{equation}}
\newcommand{\beas}{\begin{eqnarray*}}
\newcommand{\eeas}{\end{eqnarray*}}
\newcommand{\bea}{\begin{eqnarray}}
\newcommand{\eea}{\end{eqnarray}}

\numberwithin{equation}{section}
\usepackage{colortbl}
\definecolor{headerblue}{RGB}{70,139,100}
\definecolor{rowblue}{RGB}{235,242,250}
\definecolor{rowwhite}{RGB}{250,250,250}
\definecolor{highlightgreen}{RGB}{220,242,220}

\begin{document}

\title[ F\MakeLowercase {inite Order Transcendental Entire Solutions of Coupled...}]{\Large F\Large\MakeLowercase {inite Order Transcendental Entire Solutions of Coupled} F\MakeLowercase {ermat}-\MakeLowercase {Type Difference Equations in Several Complex Variables}} 

\date{}
\author[J. B\MakeLowercase{anerjee} \MakeLowercase{and} A. B\MakeLowercase{anerjee}]{J\MakeLowercase{hilik} B\MakeLowercase{anerjee}$^1$$^*$ \MakeLowercase{and} A\MakeLowercase {bhijit} B\MakeLowercase {anerjee}$^2$}

\address{$^1$Department of Mathematics, University of Kalyani, West Bengal 741235, India.}
\email{jhilikbanerjee38@gmail.com, jhilikmath24@klyuniv.ac.in}
\address{$^2$Department of Mathematics, University of Kalyani, West Bengal 741235, India.}
\email{abanerjee\_kal@yahoo.co.in, abanerjee@klyuniv.ac.in}

\renewcommand{\thefootnote}{}
\footnote{2020 \emph{Mathematics Subject Classification}: 39A45, 32H30, 39A14 and 35A20.}
\footnote{\emph{Key words and phrases}: Several complex variables, meromorphic functions, Fermat-type
equations, Nevanlinna theory, difference equations.}
\footnote{*\emph{Corresponding Author}: Jhilik Banerjee}

\renewcommand{\thefootnote}{\arabic{footnote}}
\setcounter{footnote}{0}

\begin{abstract} Motivated by recent developments in complex difference equations and Nevanlinna theory in several complex variables, we investigate finite-order transcendental entire solutions of the coupled Fermat-type difference system:
\beas
\begin{cases}
f_1^{n_1}(z)+ f_2^{m_1} \left(z+c \right) = 1,\\
f_2^{n_2}(z) + f_1^{m_2} \left(z+c\right) = 1,
\end{cases} 
\eeas where $z,c=(c_1,c_2,\ldots,c_n) \in \mathbb{C}^n$ for various choices of $n_i,m_i$, $i=1,2$.
where $n_i,m_i\in\mathbb N$ and $n_i+m_i\ge2$ $(i=1,2)$. Extending the classical investigations of Gross--Yang, Liu, Liu--Cao--Cao and more recently, Xu \emph{et al.} in one and two complex variables, to a general coupled system in $\mathbb C^n$ we establish a complete characterization of all finite-order transcendental entire solutions. We have determined that the solution structure is completely determined by the relative sizes of the exponents.
\end{abstract}

\thanks{Typeset by \AmS -\LaTeX}
\maketitle


\section{Introduction}

One of the main topics in complex difference theory is the study of entire and meromorphic solutions of nonlinear functional equations. Among them, Fermat-type functional equations have attracted considerable attention because of their rich algebraic structure and interesting growth behavior. 
Their investigation has led to important connections among Nevanlinna theory, difference equations, differential equations and value distribution theory.

The classical algebraic equation $$x^n+y^m=1$$
serves as a natural model for many analytic problems.

Therefore, the existence, nonexistence and characterization of entire and meromorphic solutions of Fermat-type equations have become important research topics in modern complex analysis.

A basic example is the Fermat-type functional equation
$$f^m(z)+g^n(z)=1.$$

Many results show that transcendental solutions of this equation are highly restricted.
For example, when $m=n\geq 4$, the equation has no transcendental meromorphic solutions, and when $m=n\geq 3$, it has no transcendental entire solutions (\cite{FG2,M1,Taylor+Wiles_AnnMath_1995,Wiles_AnnMath_1995}).
These findings demonstrate the strong rigidity of Fermat-type equations and have motivated extensive studies of their differential, difference, and differential-difference counterparts.

The functional equation $f ^m(z)+ g^n(z) = 1$ can be regarded as the Fermat-type functional equation. It has drawn considerable attention from many mathematicians in the study of Fermat-type equations. It is known that the Fermat-type equation admits no transcendental meromorphic solutions for $n=m\geq 4$ (\cite{FG2}), and no transcendental entire solutions for $n=m\geq 3$ (\cite{M1}).

\subsection{{\bf Basic notations in several complex variables}}
Recently, Nevanlinna theory in several complex variables has developed an active research area. Extensions of value distribution results from one-variable complex analysis to higher-dimensional structure remains a prior theme of research.
\par Recent investigations have revealed deep connections between Nevanlinna theory and various branches of analysis and geometry. In particular, significant progress has been made in its applications to complex geometry, normal family theory, linear partial differential equations, partial difference equations, partial differential-difference equations, and Fermat-type functional equations. 
The works in references \cite{CK1}, \cite{Dovbush-2021}, \cite{Hao-Zhang-2025}, \cite{L2}, \cite{L1}, \cite{LY11}, \cite{LS1}, \cite{LB}, \cite{FL}, \cite{Majumder-2026}-\cite{MSP}, \cite{GS1}, \cite{XC1} \cite{XW1}, \cite{XH} offer a useful framework for understanding current developments and active areas of research in Nevanlinna theory in several complex variables. These studies are largely based on higher-dimensional Nevanlinna theory and the value distribution theory of meromorphic functions in several complex variables (see \cite{TBC1,CK1,CX1,K1}).

\smallskip
We take $z=(z_1,\ldots,z_n)\in\mathbb C^n$. For any subset $A\subset \mathbb{C}^n$ and $r\ge0$, define (see \cite[pp. 6]{Stoll-1974}) $A[r]=\{z\in A:\|z\|\le r\}$,
$A(r)=\{z\in A:\|z\|<r\}$ and $A\langle r\rangle=\{z\in A:\|z\|=r\}$.
We introduce the function $\tau(z)=\|z\|^2$.
On $\mathbb{C}^n$, the exterior derivative decomposes as $d=\partial+\bar{\partial}$, and we set
$d^c=\frac{\iota}{4\pi}(\bar{\partial}-\partial)$ and $dd^c=\frac{\iota}{2\pi}\partial\bar{\partial}$ (see \cite{HLY1,WS1}).
The standard \text{K\"ahler} form on $\mathbb{C}^n$ is $\upsilon= dd^c \tau >0$.
On $\mathbb{C}^n\setminus\{0\}$, we further define
$\omega = dd^c \log \tau \ge 0$ and $\sigma = d^c\log\tau \wedge \omega^{n-1}$,
where $n=\dim(\mathbb{C}^n)$ (see \cite[pp. 6]{Stoll-1974}).

\smallskip
Let $G\neq \varnothing$ be an open subset of $\mathbb{C}^n$. Let $f$ be a holomorphic function in $\mathbb{C}^n$.
Take $a\in G$. Let $G_a$ be the connectivity component of $G$ containing $a$. Assume $f\mid_{G_a}\not\equiv 0$. Then $f$ admits a local expansion
\begin{align*}
	f(z)=\sum_{\lambda=p}^{\infty}P_\lambda(z-a),
\end{align*}
where $P_\lambda$ is homogeneous of degree $\lambda$ and $P_p\not\equiv0$. The polynomials $P_{\lambda}$ depend on $f$ and $a$ only.
The integer $\mu_f^0(a)=p$ is called the zero multiplicity of $f$ at $a$ (see \cite[pp. 12]{Stoll-1974}).

\smallskip
Let $f$ be a meromorphic function on $G$, where $G\neq \varnothing$ is an open subset of $\mathbb{C}^n$.
Take $a\in G$ and $c\in\mathbb{P}^1$. Let $G_a$ be the component of $G$ containing $a$. If $0\equiv f\mid_{G_a}\not\equiv c$, define $\mu^c_f(a)=0$. Assume $0\not\equiv f\mid_{G_a}\not\equiv c$. Then an open connected neighborhood $U$ of $a$ in $G$ and holomorphic functions $g\not\equiv 0$ and $h\not\equiv 0$ exist on $U$ such that $h. f\mid_U=g$ and $\dim g^{-1}(0)\cap h^{-1}(0)\leq n-2$, where $n=\dim(\mathbb{C}^n)$. Therefore the $c$-multiplicity of $f$ is just $\mu^c_f=\mu^0_{g-ch}$ if $c\in\mathbb{C}$ and $\mu^c_f=\mu^0_h$ if $c=\infty$. The function $\mu^c_f:G\to \mathbb{Z}$ is nonnegative and is called the $c$-divisor of $f$ (see \cite[pp. 12]{Stoll-1974}).
If $f\not\equiv0$ on each component of $G$, the divisor of $f$ is $\mu_f=\mu_f^0-\mu_f^\infty$.
The function $f$ is holomorphic if and only if $\mu_f\ge0$.

\smallskip
A function $\nu:G\to \mathbb{Z}$ is said to be a divisor if and only if for each $a\in G$ an open, connected neighborhood $U$ of $a$ in $G$ and a meromorphic function $f\not\equiv 0$ exist such that $\nu\mid_{U}=\mu_f$. A divisor $\nu:G\to \mathbb{Z}$ is non-negative if and only if for each $a\in G$ an open, connected neighborhood $U$ of $a$ in $G$ and a holomorphic function $f\not\equiv 0$ exist such that $\nu\mid_{U}=\mu_f$ (see \cite[pp. 13]{Stoll-1974}).
We denote $\operatorname{supp}\nu=\overline{\{z\in G:\nu(z)\ne0\}}$.

\smallskip
Take $0<R\leq +\infty$. Let $\nu$ be a divisor on $\mathbb{C}^n(R)$ with $A=\displaystyle \operatorname{supp}\nu$. For $t>0$, the counting function $n_{\nu}$ is defined by
\begin{align*}
	\displaystyle n_{\nu}(t)=t^{-2(m-1)}\int_{A[t]}\nu\; \upsilon^{n-1},
\end{align*}
where $n=\dim(\mathbb{C}^n)$. We know that
\begin{align*}
	\displaystyle n_{\nu}(t)=\int_{A(t)}\nu\;\omega^{n-1}+n_{\nu}(0).
\end{align*}

For $0<s<r<R$, we define the valence function of $\nu$ by
\begin{align*}
	N_{\nu}(r)=N_{\nu}(r,s)=\int_{s}^r n_{\nu}(t)\frac{dt}{t}.
\end{align*}

For $\nu=\mu_f^a$, we write $n(t,a;f)$ and $N(r,a;f)$, with the usual conventions for $a=\infty$.

A non-negative divisor $\nu:\mathbb{C}^n\to \mathbb{Z}_+$ is said to be algebraic if and only if $\nu$ is  the zero divisor of a polynomial. Thus a divisor $\nu:\mathbb{C}^n\to \mathbb{Z}_+$ is algebraic if and only if $n_{\nu}$ is bounded, which implies that $N_{\nu}=O(\log r)$ (see \cite[pp. 19]{Stoll-1974}).

\smallskip
Let $G\neq \varnothing$ be an open subset in $\mathbb{C}^n$. Let $f$ be a meromorphic function in $G$ in the sense that $f$ can be written as a quotient of two relatively prime holomorphic functions. We will write $f = (f_0, f_1)$ where $f_0 \not\equiv 0$. The standard definition of the Nevanlinna characteristic function of $f$ is given by (see \cite[pp. 16-17]{Stoll-1974})
\begin{align*}
	T_f(r,s) := \int_s^r \frac{A_f(t)}{t}\, dt,
\end{align*}
where $0 < s < r$ and
\begin{align*}
	A_f(t)
	= \frac{1}{t^{2n-2}} \int_{\mathbb{C}^n(t)} f^*(\ddot{\omega}) \wedge \upsilon^{n-1}
	= \int_{\mathbb{C}^n(t)} f^*(\ddot{\omega}) \wedge \omega^{\,n-1} + A_f(0),
\end{align*}
where $n=\dim(\mathbb{C}^n)$.
Here the pullback $f^*(\ddot{\omega})$ satisfies
\begin{align*}
	f^*(\ddot{\omega}) = dd^c \log\left( |f_0|^2 + |f_1|^2 \right)
\end{align*}
for all $z$ outside of the set of indeterminacy $I_f:= \{ z \in \mathbb{C}^n : f_0(z) = f_1(z) = 0 \}$ of $f$.

Take $a\in \mathbb{P}^1$ and $0<R\leq +\infty$. Let $f\not\equiv 0$ be a meromorphic function on $\mathbb{C}^n(R)$.
For $0<r<R$, define the compensation of $f$ for $a$ by
\begin{align*}
	m^a_f(r)=\int\limits_{\mathbb{C}^n\langle r\rangle}\log \frac{1}{||f,a||} \;\sigma,
\end{align*}
where $||f,a||$ denotes the chordal distance from $f$ to $a\in\mathbb{P}^1$.
Then the First Main Theorem of Nevanlinna theory becomes
\begin{align*}
	T_f(r)=T_f(r,s)=N_{\mu^a_f}(r,s)+m^a_f(r)-m^a_f(s),
\end{align*}
where $0<s<r$.

There is slightly different way to continue the formulation of Nevanlinna theory from here (see \cite[pp.15]{HLY1}).
Take $0<R\leq +\infty$. Let $f\not\equiv 0$ be a meromorphic function on $\mathbb{C}^n(R)$. Let $0<s<r<R$. 
Now with the help of the positive logarithm function, we define the proximity function of $f$ by
\begin{align*}
	\displaystyle m(r, f)=\int_{\mathbb{C}^n\langle r\rangle} \log^+ |f|\;\sigma \geq  0.
\end{align*}

The characteristic function of $f$ is defined by $T(r,f)=m(r,f)+N(r,f)$. We know that (see \cite[pp.15]{HLY1})
\begin{align*}
	\displaystyle T\left(r,\frac{1}{f}\right)=T(r,f)-\int_{\mathbb{C}^n\langle s\rangle} \log |f|\;\sigma.
\end{align*}

We define $m(r,a;f)=m(r,f)$ if $a=\infty$ and $m(r,a;f)=m\big(r,\frac{1}{f-a}\big)$ if $a$ is finite complex number. Now if $a\in\mathbb{C}$, then the first main theorem of Nevanlinna theory becomes $m(r,a;f)+N(r,a;f)=T(r,f) + O(1)$, where $O(1)$ denotes a bounded function when $r$ is sufficiently large.

Finally, if we compare the functions $T_f(r)$ and $T(r,f)$, then we have (see \cite[pp.19]{HLY1})
\begin{align*}
	T_f(r)=T(r,f)+O(1).
\end{align*}

Clearly $f$ is non-constant, then $T(r, f) \rightarrow \infty$ as $r \rightarrow$ $\infty$. Further $f$ is rational if and only if $T(r,f)=O(\log r)$ (see \cite[pp. 19]{Stoll-1974}).
On the other hand, if $f$ is transcendental, then 
\[\lim\limits_{r \rightarrow \infty} \frac{T(r, f)}{\log r}=+\infty.\] 

The following theorem describes the entire and meromorphic solutions in $\mathbb{C}^n$ of the Fermat-type Eq.
\bea\label{FT1} f^m+g^m=1,\;m>1.\eea

\begin{theoA} \cite[Theorem 1.3]{GS2} For $h:\mathbb{C}^n \to \mathbb{C}$ entire, the solutions of the Eq. {\em(\ref{FT1})} are characterized as follows:
	\begin{enumerate}
		\item[(a)] for $m=2$, the entire solutions are $f=\cos(h)$ and $g=\sin(h)$;
		\item[(b)] for $m>2$, there are no non-constant entire solutions;
		\item[(c)] for $m=2$, the meromorphic solutions are of the form
		$f=\frac{1 - \beta^2}{1 + \beta^2}$ and $g=\frac{2\beta}{1 + \beta^2}$,
		with $\beta$ being meromorphic on $\mathbb{C}^n$;
		\item[(d)] for $m=3$, the meromorphic solutions are of the form
		\[f = \frac{1}{2\wp(h)}\left( 1 + \frac{\wp^{(1)}(h)}{\sqrt{3}}\right)\;\;\text{and}\;\;g = \frac{1}{2\wp(h)} \left(1-\frac{\wp^{(1)}(h)}{\sqrt{3}}\right);\]
		\item[(e)] for $m>3$, there are no non-constant meromorphic solutions.
	\end{enumerate}
\end{theoA}
In 2009, Liu \cite{Liu1} proved the following two propositions:
\begin{theoB}[{\cite[Proposition 5.1]{Liu1}}]
	Let \(f\) be a non-constant finite order entire solution of the non-linear
	difference equation
	\[	f(z)^2+f(z+c)^2=a(z)^2,\]then $	f(z)=\frac{1}{2}\left(h_1(z)+h_2(z)\right),$
	where $	\frac{h_1(z+c)}{h_1(z)}=i \quad \text{and} \quad \frac{h_2(z+c)}{h_2(z)}=-i,$
	$h_1(z)h_2(z)=a(z)^2,$	where $a(z)$ is a non-vanishing small function to $f(z)$ with period $c$.
\end{theoB}
\begin{theoC}[{\cite[Proposition 5.3]{Liu1}}]
	There is no non-constant finite order entire solution of the non-linear
	difference equation
	\[f(z)^2+\left(\Delta_c f\right)^2=a^2,\] where $a$ is a non-zero constant.
\end{theoC}

In 2012, Liu-Cao-Cao \cite{LCC1} considered the following difference equation and obtained:
\begin{theoD}[{\cite[Theorem 1.1]{LCC1}}]
	The transcendental entire solutions of finite order of
	\[
	f(z)^2+f(z+c)^2=1
	\]
	must satisfy
	$f(z)=\sin(Az+B),$
	where \(B\) is a constant and $A=\frac{(4k+1)\pi}{2c},$	with \(k\in\mathbb{Z}\).
\end{theoD}
Now, recall that the pair $(f(z), g(z))$ is referred to as a set of finite-order transcendental entire solutions for the system
\[\left\{\begin{array}{l}
	f^{n_1}(z)+ g^{m_1}(z)=1 \\
	f^{n_2}(z)+ g^{m_2}(z)=1
\end{array}\right.\]
if $f(z)$ and $g(z)$ are transcendental entire functions and $\rho=\max \{\rho(f), \rho(g)\}<\infty$.
\par Recently Xu et al. \cite{XLL1} studied the existence and the forms of the finite order transcendental entire solutions to a system of Fermat-type difference equations and obtained the following result:
\begin{theoE}[{\cite[Theorem 1.2]{XLL1}}]
	Let \(c=(c_1,c_2)\) be a constant in \(\mathbb{C}^2\). Then any pair of
	transcendental entire solutions with finite order for the system of
	Fermat-type difference equations
	\bea\label{ee}\left\{\begin{aligned}
			f_1(z_1,z_2)^2+f_2(z_1+c_1,z_2+c_2)^2 &=1,\\
			f_2(z_1,z_2)^2+f_1(z_1+c_1,z_2+c_2)^2 &=1,
		\end{aligned}
		\right.\eea
	have the following forms
	\[(f_1(z),f_2(z))=	\left(\frac{e^{L(z)+B_1}+e^{-(L(z)+B_1)}}{2},
	\frac{A_{21}e^{L(z)+B_1}+A_{22}e^{-(L(z)+B_1)}}{2}\right),\]
	where $L(z)=a_1z_1+a_2z_2,$	$B_1$ is a constant in $\mathbb{C}$, and $c$, $A_{21}$, $A_{22}$
	satisfy one of the following cases:
	
	\begin{enumerate}
		\item[{\em(i)}]
		$L(c)=2k\pi i,\; A_{21}=-i,\; A_{22}=i,$
		or
		$L(c)=(2k+1)\pi i,\; A_{21}=i, \;A_{22}=-i,$ where \(k\in\mathbb{Z}\).
		
		\item[(ii)]
		$L(c)=\left(2k+\frac12\right)\pi i,\;A_{21}=-1,\;A_{22}=-1,$
		or
		$L(c)=\left(2k-\frac12\right)\pi i,\;A_{21}=1,\;A_{22}=1.$
	\end{enumerate}
\end{theoE}
The following example shows that the forms of the function $(f_1(z),f_2(z))$ is not complete.
\begin{exm}
	Consider \[(f_1(z), f_2(z))=\left(\frac{e^{z_1+z_2+z_2^2}+e^{-(z_1+z_2+z_2^2)}}{2},\frac{e^{z_1+z_2+z_2^2}+e^{-(z_1+z_2+z_2^2)}}{2} \right)\] and $c=\left(\frac{\pi}{2},0\right)$. Then, $(f_1(z),f_2(z))$ satisfy {\em(\ref{ee})} but the solutions are not in the forms mentioned in {\em Theorem E}. Here, $B_1$ as used in {\em Theorem E} is not a constant.
\end{exm}

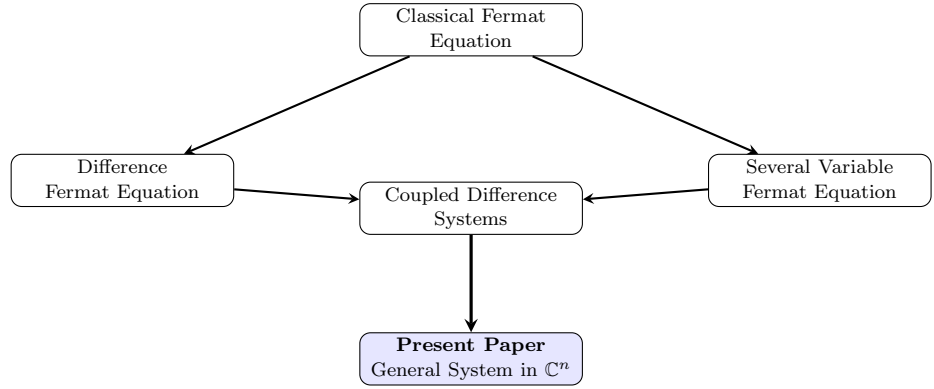
\begin{figure}[H]
	\centering
	
	\begin{tikzpicture}[
		scale=0.92,
		transform shape,
		node distance=1.5cm,
		every node/.style={
			draw,
			rounded corners,
			align=center,
			font=\scriptsize,
			minimum width=3.2cm,
			minimum height=0.75cm,
			inner sep=2pt
		},
		>=stealth
		]
		
		\node (A) {Classical Fermat\\Equation};
		
		\node (B) [below left=1.4cm and 1.8cm of A]
		{Difference\\Fermat Equation};
		
		\node (C) [below right=1.4cm and 1.8cm of A]
		{Several Variable\\Fermat Equation};
		
		\node (D) [below=1.8cm of A]
		{Coupled Difference\\Systems};
		
		\node (E) [below=1.4cm of D,
		fill=blue!10]
		{\textbf{Present Paper}\\
			General System in $\mathbb{C}^{n}$};
		
		\draw[->,thick] (A)--(B);
		\draw[->,thick] (A)--(C);
		\draw[->,thick] (B)--(D);
		\draw[->,thick] (C)--(D);
		\draw[->,very thick] (D)--(E);
		
	\end{tikzpicture}
	
	\caption{Evolution of Fermat-type functional equations leading to the present investigation.}
	\label{fig:roadmap}
	
\end{figure}
\begin{center}
	\begin{tcolorbox}[
		colback=blue!4,
		colframe=blue!40!black,
		width=0.94\linewidth,
		boxrule=0.6pt,
		arc=2mm,
		left=2mm,
		right=2mm,
		top=1mm,
		bottom=1mm
		]
		\textbf{Main novelty of the present work.}
		Unlike previous investigations that focused primarily on single Fermat-type equations or low-dimensional coupled systems, the present paper establishes a classification theory for finite-order transcendental entire solutions of a coupled Fermat-type difference system in several complex variables. The obtained results reveal new periodicity phenomena, nonexistence criteria, and explicit solution structures depending on the relative sizes of the exponents.
	\end{tcolorbox}
\end{center}

Let us consider the following system of equations:
\bea\label{DE1}
\begin{cases}
	f_1^{n_1}(z_1,z_2,\ldots,z_n)+ f_2^{m_1} \left(z_1 + c_1, z_2 + c_2, \ldots, z_n + c_n \right) = 1,\\
	f_2^{n_2}(z_1,z_2,\ldots,z_n) + f_1^{m_2} \left(z_1 + c_1, z_2 + c_2, \ldots, z_n + c_n \right) = 1
\end{cases}
\eea

The objective of the paper is to characterize all finite-order transcendental entire solutions $((f_1(z),f_2(z))$ of the system (\ref{DE1}) for different combinations of the parameters $n_1$, $n_2$, $m_1$, and $m_2$ by extending the solutions in several complex variable setup.   

\begin{theo}\label{t1} Let $c\in\mathbb{C}^n$. Then any finite order transcendental entire solutions for the system of equations (\ref{DE1}), where $m_1,m_2,n_1,n_2\in\mathbb{N}$ and $n_i+m_i \ge 2$ are characterized as follows:
\begin{enumerate}
\item[\em(1)] if $n_1=m_1$, then $m_1=n_1=m_2=n_2=2$ and
	\[(f_1(z),f_2(z))=\cos (L_1(z)+Q_1(z)),\cos (L_2(z)+Q_2(z))),\]
	where $L_1(z)=A_{11}z_1+A_{11}z_2\ldots+A_{1n}z_n$, $L_2(z)=B_{11}z_1+B_{12}z_2\ldots+B_{1n}z_n$, $Q_i(z_2,\ldots,z_n)$ is a polynomial such that $Q_i(z+2c)=Q_i(z)$ for $i=1,2$, $A_{1i},B_{1i}\in\mathbb{C}$ for $i=1,2,\ldots,n$  one of the following hold:
	\begin{enumerate}
		\item[\em(i)] $e^{2\iota (A_{11}c_1+A_{12}c_2+\ldots+A_{1n}c_n)}=-1\;\;\text{and}\;\;e^{2\iota (B_{11}c_1+B_{12}c_2+\ldots+B_{1n}c_n)}=-1$;
		
		\smallskip
		\item[\em(ii)] $e^{-2\iota (A_{11}c_1+A_{12}c_2+\ldots+A_{1n}c_n)}=1\;\;\text{and}\;\;e^{-2\iota (B_{11}c_1+B_{12}c_2+\ldots+B_{1n}c_n)}=1$;
	
\end{enumerate}

\item[\em(2)] if $n_1>m_1$, then $n_2=m_1=1$, $n_1=m_2 \ge 2$ and\begin{align*}
	(f_1(z), f_2(z))=\left(e^{a_1z_1+\ldots+a_nz_n}h_1(z),\;e^{b_1z_1+\ldots+b_nz_n}h_2(z)\right),
\end{align*}
where $(a_1,\ldots,a_n),(b_1,\ldots,b_n) \in\mathbb{C}^n$ such that $e^{2(a_1c_1+\ldots+a_nzc_n)}=t\;(t^{n_1}=1)$ and $e^{2(b_1c_1+\ldots+b_nzc_n)}=1$,  $h_1$ and  $h_2$ are  $2c$-periodic entire functions in $\mathbb{C}^n$.
 \item[\em(3)]
if $n_1<m_1$, then $n_1=m_2=1$, $m_1=n_2 \ge 2$ and \begin{align*}
	(f_1(z), f_2(z))=\left(e^{\tilde a_1z_1+\ldots+\tilde a_nz_n}\tilde h_1(z),\;e^{\tilde b_1z_1+\ldots+\tilde b_nz_n}\tilde h_2(z)\right),
\end{align*}
where $(\tilde a_1,\ldots,\tilde a_n),(\tilde b_1,\ldots,\tilde b_n) \in\mathbb{C}^n$ such that $e^{2(\tilde a_1c_1+\ldots+\tilde a_nzc_n)}=1$ and $e^{2(\tilde b_1c_1+\ldots+\tilde b_nzc_n)}=t\;(t^{m_1}=1)$,  $\tilde h_1$ and $\tilde h_2$ are  $2c$-periodic entire functions in $\mathbb{C}^n$.
\end{enumerate}
\end{theo}
\setlength{\arrayrulewidth}{1pt}
\begin{table}[H]
	\centering
	\footnotesize
	\caption{Comparison of representative results on Fermat-type equations.}
	\label{tab:literature}
	
	\renewcommand{\arraystretch}{1.35}
	
	\begin{tabular}{|p{2.8cm}|p{1.7cm}|p{3.0cm}|p{5.3cm}|}
		\hline
		\rowcolor{headerblue}
		\color{white}\textbf{Reference}
		&
		\color{white}\textbf{Dimension}
		&
		\color{white}\textbf{\hspace{.5cc}Equation Type}
		&
		\color{white}\textbf{\hspace{3.5cc}Main Result}
		\\
		\hline
		
		\rowcolor{rowblue}
		Gross--Yang
		&
		$\mathbb C$
		&
		$f^m+g^m=1$
		&
		Classification of solutions.
		\\
		\hline
		
		\rowcolor{rowwhite}
		Liu (2009)
		&
		$\mathbb C$
		&
		Difference equations
		&
		Nonexistence and structure results.
		\\
		\hline
		
		\rowcolor{rowblue}
		Liu--Cao--Cao (2012)
		&
		$\mathbb C$
		&
		Difference Fermat equations
		&
		Explicit finite-order solutions for a single equation.
		\\
		\hline
		
		\rowcolor{rowwhite}
		Xu et al. (2020)
		&
		$\mathbb C^2$
		&
		Coupled PDE--difference systems
		&
		Characterization theorem for a special coupled system.
		\\
		\hline
		
		\rowcolor{highlightgreen}
		\textbf{Present paper}
		&
		\textbf{$\mathbb C^n$}
		&
		\textbf{System (\ref{DE1})}
		&
		\textbf{General coupled system with complete classification of finite-order entire solutions.}
		\\
		\hline
	\end{tabular}
\end{table}

\section{\bf{Key lemmas}}
\begin{lem}\label{L.2} \cite[Lemma 1.2]{HY1} Let $f$ be a non-constant meromorphic function in $\mathbb{C}^n$ and let $a_1,a_2,\ldots,a_q$ be different points in $\mathbb{C}\cup \{\infty\}$. Then
	\beas \parallel (q-2)T(r,f)\leq \sideset{}{_{j=1}^{q}}{\sum} \ol N(r,a_j;f)+O(\log (rT(r,f))).\eeas
\end{lem}

\begin{lem}\label{L.3} \cite[Theorem 1.26]{HLY1} Let $f$ be non-constant meromorphic function in $\mathbb{C}^n$. Assume that
	$R(z, w)=\frac{A(z, w)}{B(z, w)}$. Then
	\beas T\left(r, R_f\right)=\max \{p, q\} T(r, f)+O\Big(\sideset{}{_{j=0}^p}{\sum} T(r, a_j)+\sideset{}{_{j=0}^q}{\sum}T(r, b_j)\Big),\eeas
	where $R_f(z)=R(z, f(z))$ and two coprime polynomials $A(z, w)$ and $B(z,w)$ are given
	respectively as follows:
	\[A(z,w)=\sideset{}{_{j=0}^p}{\sum} a_j(z)w^j\;\;\text{and}\;\;B(z,w)=\sideset{}{_{j=0}^q}{\sum} b_j(z)w^j.\]
\end{lem}

\begin{lem}\label{L.6}\cite[Lemma 3.2]{HLY1} Let $f_j\not\equiv 0\;(j=1,2,\ldots,n;n\geq 3)$ be meromorphic functions in $\mathbb{C}^n$ such that $f_1,\ldots,f_{n-1}$ are non-constants and $f_1+\cdots+f_n=1$. If
	\beas \parallel \;\;\sideset{}{_{j=1}^n}{\sum}\Big\lbrace N_{n-1}(r,0;f_j)+(n-1)\ol{N}(r,f_j)\Big\rbrace<\lambda T(r,f_j)+O(\log^+(T(r,f_j))\eeas
	holds for $j=1,2,\ldots,n-1$, where $\lambda<1$, then $f_n\equiv 1$.
\end{lem}

\begin{lem}\label{L.5}\cite[Proposition 3.2]{hy1} Let $P$ be a non-constant entire function in $\mathbb{C}^n$. Then
	\[\rho(e^P)=
	\begin{cases}
		\deg(P) & \text{if $P$ is a polynomial,}\\
		+\infty & \text{otherwise}.
	\end{cases}\]
\end{lem}

\begin{lem}\label{L.7}\cite[Theorem 2.2]{CX1} Let $f$ be a non-constant meromorphic function in $\mathbb{C}^n$ and let $c\in \mathbb{C}^n$. If
	\bea\label{Ss}\limsup\limits_{r\rightarrow \infty} \frac{\log T(r,f)}{r}=0,\eea
	then
	\begin{align*}
		\parallel\;T(r,f(z+c))=T(r,f)+o(T(r,f)).
	\end{align*}
\end{lem}

\begin{lem}\label{ARK}\cite[Theorem 2.1]{Kramer} If $f$ is an entire function in $\mathbb{C}^{m} (m \geq 1)$, 
	not identically zero, then $\displaystyle f = P e^{h}$, where $P$ is a polynomial and $h$ is entire, if and only 
	if the intersection of $Z(f)$ with some $\Delta_{r}$ is compact, where $Z(f)$ is the set of zeros of $f$ and 
	$\Delta_{r} = \{ z \in \mathbb{C}^{m} : r_{1}|z_{1}| = \cdots = r_{m}|z_{m}|\}$, $r_{1}, r_{2}, \ldots , r_{m}$ 
	are positive constants. \end{lem}
\begin{rem}\label{R1} If an entire function $f$ does not have any zero, then $Z(f) = \emptyset$ and so its 
	intersection with any $\Delta_{r}$ is compact. Hence, in this case $\displaystyle f = e^{h}$ for some entire 
	function $h$. \end{rem}

\section {{\bf Proof of Theorem \ref{t1}}}
\begin{proof} Let $(f_1(z),f_2(z))$ be a pair of finite order non-constant entire solution of the system of equations (\ref{DE1}). Now using Lemma \ref{L.3} to the system of equations (\ref{DE1}), we get
	\bea\label{e3.1} \parallel\;n_i T(r,f_i)+o(T(r,f_i))=m_i T(r,f_j(z+c))+o(T(r,f_j(z+c))),\eea
	where $i,j\in\{1,2\}$ such that $i\neq j$. Since $\rho(f_i)<+\infty$, we have
	\[\limsup_{r \to \infty} \frac{\log T(r,f_i)}{r}=0\]
	for $i=1,2$ and so by Lemma \ref{L.7}, we have
	\bea\label{e3.2} \parallel\;T(r,f_i(z+c))=T(r,f_i)+o(T(r,f_i)),\eea
	where $i\in\{1,2\}$. Now from (\ref{e3.1}) and (\ref{e3.2}), we get
	\bea\label{e3.3} \parallel\;n_i T(r,f_i)+o(T(r,f_i))=m_i T(r,f_j)+o(T(r,f_j)),\eea
	where $i,j\in\{1,2\}$ such that $i\neq j$. Hence (\ref{e3.3}) shows that $o(T(r,f_i))$ can be replaced by $o(T(r,f_j))$, where $i,j\in\{1,2\}$ such that $i\neq j$. 
	
	On the other hand, using Lemma \ref{L.3} and (\ref{e3.2}) to (\ref{DE1}), we obtain
	\bea\label{lm.2}\parallel\;m_i T(r, f_j)&=& m_i T(r, f_j(z+c))+o(T(r, f_j))\\
	& =&T\left(r, f_j^{m_i}(z+c)\right)+o(T(r, f_j))\nonumber\\
	& =&T\left(r,f_i^{n_i}\right)+o(T(r, f_j)) \nonumber\\
	& =&n_i T(r, f_i)+o(T(r, f_j)),\nonumber\eea
	where $i,j\in\{1,2\}$ such that $i\neq j$.\par
	
	\medskip
	
	First we suppose that $m_1m_2-n_1n_2>0$. Then from  (\ref{lm.2}), one can easily deduce that
	\[\parallel\;(m_1m_2-n_1n_2)T(r,f_j)= o(T(r, f_j)),\]
	which is impossible, where $j\in\{1,2\}$. Therefore in this case, the system of equations (\ref{DE1}) do not have any finite order non-constant entire solutions.\par

	\medskip\medskip
	Next we suppose that $m_1m_2-n_1n_2\leq 0$. If $n_1n_2-m_1m_2>0$, then from  (\ref{lm.2}), we get
	\[\parallel\;(n_1n_2-m_1m_2)T(r,f_i)= o(T(r, f_i)),\]
	which is impossible, where $i\in\{1,2\}$. So the system of equations (\ref{DE1}) do not have any finite order non-constant entire solutions. Hence 
	\bea\label{e3.4} n_1n_2=m_1m_2.\eea
	
	Let
	\begin{align}\label{lm.3} 
		h_i(z)=\frac{f_i^{n_i}(z)-1}{f_i^{n_i}(z)},
	\end{align}
	where $i\in\{1,2\}$. Clearly $h_i$ is a non-constant meromorphic function in $\mathbb{C}^n$. Using Lemma \ref{L.3} to (\ref{lm.3}), we get
	\begin{align}\label{lm.4}
		\parallel T(r,h_i)+o(T(r,h_i))=n_iT(r,f_i)+o(T(r,f_i)),
	\end{align}
	where $i\in\{1,2\}$. Now using (\ref{DE1}) to (\ref{lm.3}), one can easily deduce that
	\begin{align*}
		\parallel\;\ol N(r,h_i)\leq \ol N(r,0, f_i^{n_i})=\ol N(r,0,f_i)+o(T(r,f_i)),
	\end{align*}
	\begin{align*}
		\parallel\;\ol N(r,0,h_i)=\ol N(r,1,f_i^{n_i})\leq \ol N(r,0,f_j^{m_i}(z+c))=\ol N(r,0,f_j(z+c))+o(T(r,f_j(z+c)))
	\end{align*}
	and $\parallel \ol N(r,1,h_i)=0$,
	where $i,j\in\{1,2\}$. Therefore in view of the first main theorem and using Lemma \ref{L.2} and (\ref{lm.4}), we get
	\begin{align}\label{lm.6} 
		\parallel\;n_i T(r,f_i)&=T(r,h_i)+o(T(r,h_i))\\&\leq \ol N(r,h_i)+\ol N(r,0,h_i)+\ol N(r,1,h_i)+o(T(r,h_i))\nonumber\\&\leq
		\ol N(r,0,f_i)+\ol N(r,0,f_i(z+c))+o(T(r,f_i))+o(T(r,f_i(z+c)))\nonumber\\&\leq
		T(r,f_i)+T(r,f_j(z+c))+o(T(r,f_i))+o(T(r,f_j(z+c))),\nonumber
	\end{align}
	where $i,j\in\{1,2\}$. Using (\ref{e3.1})-(\ref{e3.3}) to (\ref{lm.6}), we get
	\begin{align*}
		\parallel\;\left(n_i-1-\frac{n_i}{m_i}\right)T(r,f_i)\leq o(T(r,f_i)),
	\end{align*}
	where $i\in\{1,2\}$ and so 
	\begin{align}\label{e3.7}
		\frac{1}{n_i}+\frac{1}{m_i}\geq 1
	\end{align}
	for  $i\in\{1,2\}$.
	
	The constraints (\ref{e3.4}) and (\ref{e3.7})
	reduce the proof to three mutually exclusive
	parameter regimes. Their logical relationship
	is summarized in Figure~\ref{fig:classification}.  \par
	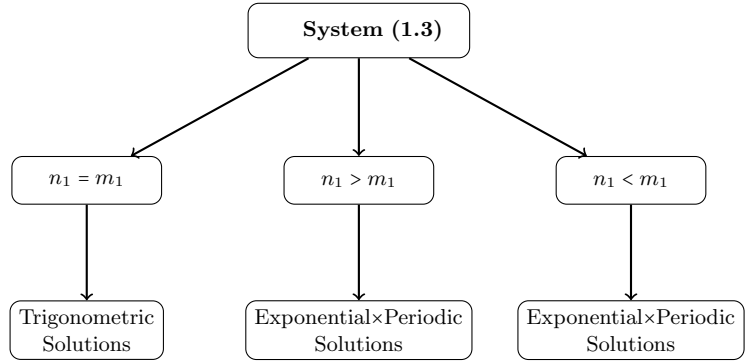
\begin{figure}[H]
		\centering
		
		\begin{tikzpicture}[
			scale=0.9,
			transform shape,
			every node/.style={
				draw,
				rounded corners,
				align=center,
				font=\footnotesize,
				minimum width=2.2cm,
				minimum height=0.7cm
			}
			]
			
		\node[
		draw,
		rounded corners,
		text width=3.0cm,
		minimum height=0.8cm,
		align=center
		] (A)
		{\textbf{\;\;\;\;System (\ref{DE1})}};
			
			\node (B) at (-4,-2.2)
			{$n_1=m_1$};
			
			\node (C) at (0,-2.2)
			{$n_1>m_1$};
			
			\node (D) at (4,-2.2)
			{$n_1<m_1$};
			
			\node (E) at (-4,-4.4)
			{Trigonometric\\Solutions};
			
			\node (F) at (0,-4.4)
			{Exponential$\times$Periodic\\Solutions};
			
			\node (G) at (4,-4.4)
			{Exponential$\times$Periodic\\Solutions};
			
			\draw[->,thick] (A)--(B);
			\draw[->,thick] (A)--(C);
			\draw[->,thick] (A)--(D);
			
			\draw[->,thick] (B)--(E);
			\draw[->,thick] (C)--(F);
			\draw[->,thick] (D)--(G);
			
		\end{tikzpicture}
		
		\caption{Classification structure established in Theorem \ref{t1}.}
		\label{fig:classification}
		
	\end{figure}
	
	\medskip
	{\bf Case 1.} Let $n_1= m_1$. Then (\ref{e3.4}) gives  $n_2=m_2$. Since $n_i+m_i> 2$ for $i=1,2$, it follows from (\ref{e3.7}) that $m_1=n_1=m_2=n_2=2.$  Then from (\ref{DE1}), we get 
	\bea\label{e3.10} (f_1(z)+\iota f_2(z+c))(f_1(z)-\iota f_2(z+c))=1\eea
	and
	\bea\label{e3.11} (f_2(z)+\iota f_1(z+c))(f_2(z)-\iota f_1(z+c))=1.\eea
	
	Clearly from (\ref{e3.10}) and (\ref{e3.11}), we see that the functions $f_1(z)+\iota f_2(z+c)$, $f_1(z)-\iota f_2(z+c)$, $f_2(z)+\iota f_1(z+c)$ and $f_2(z)-\iota f_1(z+c)$ have no zeros. In view of Lemma \ref{ARK} and Remark \ref{R1}, we may assume that
	\bea\label{e3.12} f_1(z)+\iota f_2(z+c)=e^{\iota P_1(z)},\eea
	\bea\label{e3.13} f_1(z)-\iota f_2(z+c)=e^{-\iota P_1(z)},\eea
	\bea\label{e3.14} f_2(z)+\iota f_1(z+c)=e^{\iota P_2(z)}\eea
	and
	\bea\label{e3.15} f_2(z)-\iota f_1(z+c)=e^{-\iota P_2(z)},\eea
	where $P_1(z)$ and $P_2(z)$ are entire functions in $\mathbb{C}^n$. Note that $\rho(f_i)<+\infty$ and $\rho(f_i(z+c))<+\infty$ for $i=1,2$. Now from (\ref{e3.12}) and (\ref{e3.14}), we see that $\rho\left(e^{\iota P_1(z)}\right)<+\infty$ and $\rho\left(e^{\iota P_2(z)}\right)<+\infty$ and so by Lemma \ref{L.5}, we deduce that  $P_1(z)$ and $P_2(z)$ are polynomials in $\mathbb{C}^n$.

	Solving (\ref{e3.12}) and (\ref{e3.13}), we get
	\bea\label{e3.16} f_1(z)=\frac{e^{\iota P_1(z)}+e^{-\iota P_1(z)}}{2}=\cos (P_1(z))\eea
	and
	\bea\label{e3.17} f_2(z+c)=\frac{e^{\iota P_1(z)}-e^{-\iota P_1(z)}}{2\iota}=\sin (P_1(z)).\eea
	
	Again solving (\ref{e3.14}) and (\ref{e3.15}), we obtain
	\bea\label{e3.18} f_2(z)=\frac{e^{\iota P_2(z)}+e^{-\iota P_2(z)}}{2}=\cos (P_2(z))\eea
	and
	\bea\label{e3.19} f_1(z+c)=\frac{e^{\iota P_2(z)}-e^{-\iota P_2(z)}}{2\iota}=\sin (P_2(z)).\eea
	
	Since $f_1$ and $f_2$ are non-constant, from (\ref{e3.16}) and (\ref{e3.18}), we deduce that both $P_1$ and $P_2$ are non-constant polynomials in $\mathbb{C}^n$. 
	
	Now from (\ref{e3.16}) and (\ref{e3.19}), we have
	\bea\label{e3.20} \iota e^{\iota{(P_1(z+c)-P_2(z))}}+\iota e^{-\iota{(P_1(z+c)+P_2(z))}}+e^{-2\iota P_2(z)}=1.\eea
	
	Again from (\ref{e3.17}) and (\ref{e3.18}), we get
	\bea\label{e3.21} \iota e^{\iota {(P_2(z+c)-P_1(z))}}+\iota e^{-\iota{(P_2(z+c)+P_1(z))}}+e^{-2\iota P_1(z)}=1.\eea
	
	Thus, using Lemma \ref{L.6} to (\ref{e3.20}) and (\ref{e3.21}), we have respectively
	\beas \text{either}\;\;  \iota e^{\iota{(P_1(z+c)-P_2(z))}}=1 \;\;\text{or}\;\;\iota e^{\iota {(P_1(z+c)+P_2(z))}}=1.\eeas
	and
	\beas \text{either}\;\;\iota e^{\iota {(P_2(z+c)-P_1(z))}}=1 \;\;\text{or}\;\ \iota e^{-\iota{(P_2(z+c)+P_1(z))}}=1.\eeas

The exponential equalities extracted from (\ref{e3.20}) and (\ref{e3.21}) produce four algebraic alternatives; some lead to contradictions while others produce admissible cosine-type solutions. For reader convenience, we summarize these branches and their outcomes in the following boxed logic tree:

		\begin{figure}[htbp]
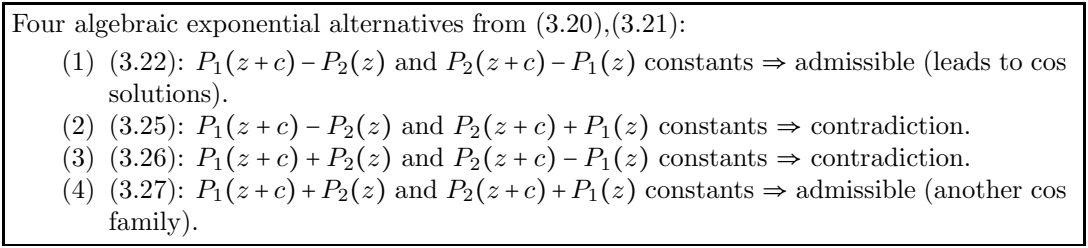

	\centering
	\fbox{%
		\begin{minipage}{0.9\linewidth}
			\small
			Four algebraic exponential alternatives from (3.20),(3.21):
			\begin{enumerate}
				\item (3.22): $P_1(z+c)-P_2(z)$ and $P_2(z+c)-P_1(z)$ constants $\Rightarrow$ admissible (leads to cos solutions).
				\item (3.25): $P_1(z+c)-P_2(z)$ and $P_2(z+c)+P_1(z)$ constants $\Rightarrow$ contradiction.
				\item (3.26): $P_1(z+c)+P_2(z)$ and $P_2(z+c)-P_1(z)$ constants $\Rightarrow$ contradiction.
				\item (3.27): $P_1(z+c)+P_2(z)$ and $P_2(z+c)+P_1(z)$ constants $\Rightarrow$ admissible (another cos family).
			\end{enumerate}
		\end{minipage}
	} \caption{Logic tree for Sub-cases in Case 1.}
	\label{fig:subcases}
\end{figure}
In view of above, we consider the following four sub-cases.\par
	\medskip
	{\bf Sub-case 1.1.} Let
	\bea\label{e3.22} \begin{cases}  \iota e^{\iota {(P_1(z+c)-P_2(z))}}=1,\\ \iota e^{\iota{(P_2(z+c)-P_1(z))}}=1.\end{cases}\eea
	
	It follows from (\ref{e3.22}) that $P_1(z+c)-P_2(z)$ and $P_2(z+c)-P_1(z)$ are both constants. Consequently both $P_1(z+2c)-P_1(z)$ and $P_2(z+2c) -P_2(z)$ are constants.
	Since both $P_1(z+2c)-P_1(z)$ and $P_2(z+2c) -P_2(z)$ are constants, it follows we may assume that
	\bea\label{e3.24}
	\begin{cases}
		P_1(z)=A_{11}z_1+A_{12}z_2+\ldots+A_{1n}z_n+ Q_1(z_1,z_2,\ldots,z_n), \\
		P_2(z)=B_{11}z_1+B_{12}z_2+\ldots+B_{1n}z_n+ Q_2(z_1,z_2,\ldots,z_n),
	\end{cases}
	\eea
	where $A_{1i}, B_{1i}\in\mathbb{C}$ for $i=1,2,\ldots,n$ and $ Q_i(z_1,z_2,\ldots,z_n)$ is a polynomial in $\mathbb{C}^{n}$ such that
	$\ Q_i(z+2c)= Q_i(z)$ for $i=1,2$. On the other hand, (\ref{e3.22}) gives
	\[e^{i(P_1(z+2c)-P_1(z))}=-1 \;\; {\text{and}} \;\;e^{i(P_2(z+2c)-P_2(z))}=-1.\]
	
	Clearly from (\ref{e3.24}), we have respectively
	\[e^{2\iota (A_{11}c_1+A_{12}c_2+\ldots+A_{1n}c_n)}=-1\;\;\text{and}\;\;e^{2\iota (B_{11}c_1+B_{12}c_2+\ldots+B_{1n}c_n)}=-1.\]
	
	Finally, from (\ref{e3.16}) and (\ref{e3.18}), we may assume that
	\[(f_1(z),f_2(z))=\left(\cos (L_1(z)+Q_1(z)),\;\cos (L_2(z)+Q_2(z)\right)),\]
	where $L_1(z)=A_{11}z_1+A_{12}z_2+\ldots+A_{1n}z_n$, $L_2(z)=B_{11}z_1+B_{12}z_2+\ldots+B_{1n}z_n$, $A_{1i},B_{1i}\in\mathbb{C}$ for $i=1,2,\ldots,n$ such that $e^{2\iota (A_{11}c_1+A_{12}c_2+\ldots+A_{1n}c_n)}=-1$, $e^{2\iota (B_{11}c_1+B_{12}c_2+\ldots+B_{1n}c_n)}=-1$ and
	$Q_i(z_1,z_2,\ldots,z_n)$ is a polynomial in $\mathbb{C}^n$ such that $Q_i(z+2c)=Q_i(z)$ for $i=1,2$.

	\medskip
	{\bf Sub-case 1.2.} Let
	\bea\label{e3.25} \begin{cases}  \iota e^{\iota{(P_1(z+c)-P_2(z))}}=1,\\\iota e^{-\iota{(P_2(z+c)+P_1(z))}}=1.\end{cases}\eea
	
	Clearly from (\ref{e3.25}), it follows that $P_1(z+c)-P_2(z)$ and $P_2(z+c)+P_1(z)$ are both constants. This means $P_2(z+2c)+P_2(z)$ is also a constant, which contradicts the fact that $P_2(z)$ is a non-constant polynomial.\par

	\medskip
	{\bf Sub-case 1.3.} Let
	\bea\label{e3.26} \begin{cases} \iota e^{-\iota{(P_1(z+c)+P_2(z))}}=1,\\ \iota e^{\iota{(P_2(z+c)-P_1(z))}}=1.\end{cases}\eea
	
	We deduce from (\ref{e3.26}) that $P_1(z+c)+P_2(z)$ and $P_2(z+c)-P_1(z)$ are both constants. This means $P_1(z+2c)+P_1(z)$ is also a constant, which is a contradiction.\par

	\medskip
	{\bf Sub-case 1.4.} Let
	\bea\label{e3.27} \begin{cases}  \iota e^{-\iota{(P_1(z+c)+P_2(z))}}=1,\\ \iota e^{-\iota{(P_2(z+c)+P_1(z))}}=1.\end{cases}\eea
	
	Clearly from (\ref{e3.27}), it follows that $P_1(z+c)+P_2(z)$ and $P_2(z+c)+P_1(z)$ are both constants. Consequently $P_1(z+2c)-P_1(z)$ and $P_2(z+2c) -P_2(z)$ are also constants.
	Clearly from (\ref{e3.27}), we have $e^{-\iota (P_i(z+2c)-P_i(z))}=1$ for $i=1,2$. Now proceeding in the same way as done in the proof of Sub-case 1.1, we can conclude that
	\[(f_1(z),f_2(z))=\left(\cos (\tilde L_1(z)+R_1(z)),\;\cos (\tilde L_2(z)+R_2(z))\right),\]
	where $\tilde L_1(z)=\tilde A_{11}z_1+\tilde A_{12}z_2+\ldots+\tilde A_{1m}z_m$, $\tilde L_2(z)=\tilde B_{11}z_1+\tilde B_{12}z_2+\ldots+\tilde B_{1n}z_n$, $\tilde A_{1i}, \tilde B_{1i}\in\mathbb{C}$, $i=1,2,\ldots,n$ such that $e^{-2\iota (\tilde A_{11}c_1+\tilde A_{12}c_2+\ldots+\tilde A_{1n}c_n)}=1$, $e^{-2\iota (\tilde B_{11}c_1+\tilde B_{12}c_2+\ldots+\tilde B_{1n}c_n)}=1$ and $R_i(z)$ is a polynomial in $\mathbb{C}^n$ such that 
	$R_i(z+2c)=R_i(z)$ for $i=1,2$.

	\medskip
	\par{\bf Case 2.} Let $n_1 > m_1$. Then from (\ref{e3.7}), we have $m_1=1$ and $n_1\geq 2$. Also from (\ref{e3.4}), we have $m_2=n_1n_2\geq 2$ and so from (\ref{e3.7}), we get $n_2=1$. Consequently $m_2=n_1\geq 2$. Now (\ref{DE1}) gives us \beas\begin{cases}
		f_1^{n_1}(z)+f_2(z+c)=1 \\ f_2(z)+f_1^{n_1}(z+c)=1.
	\end{cases} \eeas 
	
	By a simple calculation we can see that $f_1^{n_1}(z+2c)=f_1^{n_1}(z)$ and $f_2(z+2c)=f_2(z)$. Consequently 
	\bea\label{xx}
	f_1(z+2c)=tf_1(z),
	\eea
	where $t^{n_1}=1$.
	Let $f_1$ and $g_1$ be any two solutions of (\ref{xx}). If we take $h_1=\dfrac{f_1}{g_1}$ , then from (\ref{xx}), we get $h_1(z+2c) = h_1(z)$ for all $z\in\mathbb{C}^n$ and so $h_1$ is a $2c$-periodic entire function in $\mathbb{C}^n$. On the other hand 
	\[g_1(z)=e^{a_1z_1+\ldots+a_nz_n},\]
	where $(a_1,\ldots,a_n)\in\mathbb{C}^n$ such that $e^{2(a_1c_1+\ldots+a_nc_n)}=t$ is a solution of (\ref{xx}). Hence the solution of (\ref{xx}) is of the form \[f_1(z)=e^{a_1z_1+\ldots+a_nz_n}h_1(z),\]
	where $h_1$ is $2c$-periodic entire function in $\mathbb{C}^n$. Since $f_2(z+2c)=f_2(z)$, we may take \[f_2(z)=e^{b_1z_1+\ldots+b_nz_n}h_2(z),\]
	where $(b_1,\ldots,b_n)\in\mathbb{C}^n$ such that $e^{2(b_1c_1+\ldots+b_nc_n)}=1$ and $h_2$ is $2c$-periodic entire function in $\mathbb{C}^n$, Finally, we have
	\begin{align*}
		(f_1(z), f_2(z))=\left(e^{a_1z_1+\ldots+a_nz_n}h_1(z),\;e^{b_1z_1+\ldots+b_nz_n}h_2(z)\right),
	\end{align*}
	where $(a_1,\ldots,a_n),(b_1,\ldots,b_n) \in\mathbb{C}^n$ such that $e^{2(a_1c_1+\ldots+a_nc_n)}=t\;(t^{n_1}=1)$ and $e^{2(b_1c_1+\ldots+b_nzc_n)}=1$,  $h_1$ and  $h_2$ are  $2c$-periodic entire functions in $\mathbb{C}^n$.
	

	\medskip
	\par{\bf Case 3.} Let $n_1<m_1$. Then from (\ref{e3.7}), we have $n_1=1$ and $m_1\geq 2$. Also from (\ref{e3.4}), we have $n_2=m_1m_2\geq 2$ and so from (\ref{e3.7}), we get $m_2=1$. Consequently $n_2=m_1\geq 2$. Now (\ref{DE1}) gives us
	\bea\label{eq3.28}\begin{cases}
		f_1(z)+f_2^{m_1}(z+c)=1 \\ f_2^{m_1}(z)+f_1(z+c)=1.
	\end{cases} \eea 
	
	By a simple calculation we can see that $f_2^{m_1}(z+2c)=f_2^{m_1}(z)$ and $f_2(z+2c)=f_2(z)$. Now proceeding in the same way as done in the proof of Case 2, 
	we have
	\begin{align*}
		(f_1(z), f_2(z))=\left(e^{\tilde a_1z_1+\ldots+\tilde a_nz_n}\tilde h_1(z),\;e^{\tilde b_1z_1+\ldots+\tilde b_nz_n}\tilde h_2(z)\right),
	\end{align*}
	where $(\tilde a_1,\ldots,\tilde a_n),(\tilde b_1,\ldots,\tilde b_n) \in\mathbb{C}^n$ such that $e^{2(\tilde a_1c_1+\ldots+\tilde a_nc_n)}=1$ and $e^{2(\tilde b_1c_1+\ldots+\tilde b_nzc_n)}=t\;(t^{m_1}=1)$,  $\tilde h_1$ and $\tilde h_2$ are  $2c$-periodic entire functions in $\mathbb{C}^n$.
\end{proof}

\vspace{0.1in}
{\bf Compliance of Ethical Standards:}\par

{\bf Conflict of Interest.} The authors declare that there is no conflict of interest regarding the publication of this paper.\par

{\bf Data availability statement.} Data sharing not applicable to this article as no data sets were generated or analysed during the current study.

\end{document}